\documentclass[reqno,10pt,oneside]{amsart}

\usepackage{amsthm,amsmath,amssymb,amscd}
\usepackage{mathrsfs}
\usepackage{lscape}
\usepackage{graphics} 
\usepackage[all,cmtip]{xy}
\usepackage{graphicx}
\usepackage{subfigure}
\usepackage{url}	
\usepackage{hyperref}
\usepackage{color}
\usepackage[usenames,dvipsnames]{xcolor}
\usepackage{verbatim}
\usepackage{fancyhdr}
\usepackage{geometry}
\usepackage{cancel}
\usepackage[normalem]{ulem}
\usepackage{mathtools}
\usepackage[english]{babel}
\usepackage{enumitem}

\mathtoolsset{showonlyrefs,showmanualtags}
\numberwithin{equation}{section}

\usepackage[latin1]{inputenc}
\newtheorem{thm}{Theorem}[section]
\newtheorem{prop}[thm]{Proposition}
\newtheorem{lem}[thm]{Lemma}
\theoremstyle{definition}\newtheorem{defn}[thm]{Definition}
\theoremstyle{definition}
\theoremstyle{definition}\newtheorem{prob}[thm]{Problem}
\theoremstyle{definition}
\theoremstyle{definition}\newtheorem{lemma}[thm]{Lemma}
\newtheorem{cor}[thm]{Corollary}
\theoremstyle{remark}\newtheorem{rmk}[thm]{Remark}
\newtheorem{eg}{Example}[section]
\DeclareMathOperator \Vol{Vol}

\DeclareMathOperator \Id{Id}

\DeclareMathOperator \tr{tr}

\def\iu {\sqrt{-1}}
\def\pbp {\partial\bar\partial}

\begin{document}
\title{Gromov-Hausdorff limits and Holomorphic isometries}
\author[Claudio Arezzo] {Claudio Arezzo}
\address{ICTP Trieste and Univ. of Parma, arezzo@ictp.it }
\author[Chao Li] {Chao Li}
\address{ICTP Trieste, cli1@ictp.it, leecryst@mail.ustc.edu.cn}
\author[Andrea Loi]{Andrea Loi}
\address{Dipartimento di Matemantica, Universita di Cagliari, loi@unica.it}

\begin{abstract}
The aim of this paper is to study pointed Gromov-Hausdorff Convergence of sequences of K\"ahler submanifolds of a fixed 
K\"ahler ambient space. Our result shows that lower bounds on the scalar curvature imply convergence to a smooth K\"ahler
manifold satisfying the same curvature bounds, and admitting a holomorphic isometry  in the same ambient space.
We then apply this convergence result to prove that 
there are no holomorphic isometries  of a non-compact complete K\"ahler manifold with asymptotically 
non-negative  ones into a finite dimensional
complex projective space endowed with the  Fubini-Study metric.
\end{abstract}

\thanks{The second author is partially supported by NSFC, No. 12001512.}

\maketitle

\tableofcontents
\section{Introduction}

Starting from the seminal work of Gromov (\cite{Grmv81}),  Anderson (\cite{And90}), Cheeger-Colding (\cite{CC97,CC00a,CC00b}) and Cheeger-Colding-Tian (\cite{CCT02}) a great deal of work
has been devoted to understand the structure of Gromov-Hausdorff limits of algebraic varieties with various riemannian bounds typically on Ricci curvature, volumes of geodesic balls and injectivity radii or diameter.
Thanks to groundbreaking results by Donaldson-Sun (\cite{DS14}, \cite{DS17}) and the subsequent work by many authors in various different geometric situations (see, for example, 
Odaka-Spotti-Sun \cite{OSS16}, Spotti-Sun \cite{SS17}, Liu (\cite{LG18}), 
Liu and  Sz\'ekelyhidi \cite{LS22} and Sun \cite{S23}) we have now a much better understanding of which singular spaces can indeed appear in the GH-compactifications of algebraic varieties 
with riemannian bounds and how singularities 
get formed along the degenerations.
The main motivation and the largest part of these works are in fact devoted to  the study of compactifications of moduli of K\"ahler-Einstein manifolds.   On the one hand,  Tian (\cite{Tian90}), Ruan (\cite{R98}), Catlin (\cite{Cat99}) 
and Zeldtich (\cite{Zel98}) proved that 
any polarized metric on a compact K\"ahler manifold is the $\mathcal{C^{\infty}}$-limit of (suitable rescaled) projectively induced metrics in ambient spaces of diverging dimension.
On the other hand, it is conjectured  (see e.g. \cite{LSZ22}) that the Einstein metrics at study are in fact {\em never} induced by the ambient geometry of the projective space $\mathbb CP^N$ with its canonical Fubini-Study metric $\omega_{FS}$
(for the proof of this Conjecture  in some special cases, see for example \cite{Chern67, Smy67, Tsu86}), with the exceptions of 
standard holomorphic  isometric embedding of homogeneous manifolds of positive curvature, i.e. polarized flag manifolds (\cite{arlcomm}).

It was therefore natural to ask whether one could understand the (pointed) GH-limits of sequences of holomorphic submanifolds of a given ambient space (not necessarily the projective space, even though this
would be the main source of interest to us) equipped with the induced metrics. The class of manifolds and immersions analyzed in this paper is the following

\begin{defn}\label{deffond}
Given $n\in\mathbb Z_{>0}$, $R_0\in\mathbb R$ and a compact connected K\"ahler manifold $(\hat M,\hat\omega)$ of dimension $N>n$, we define
$\mathcal K(n, R_0, \hat M,\hat\omega)$ to be the class of $4$-tuple $(M,\omega,p,\phi)$ consisting of:
\begin{itemize}
\item  a smooth complete connected  K\"ahler manifold $(M,\omega)$  of dimension $n$ with scalar curvature $R(\omega)\geq R_0$;
\item  a distinguished point $p\in M$;
\item  a holomorphic isometry $\phi:(M,\omega)\rightarrow (\hat M,\hat\omega)$.
\end{itemize}
\end{defn}

Our first theorem then reads:

\begin{thm}\label{cptness2}
Every sequence $\{(M_i,\omega_i,p_i,\phi_i)\} \subset \mathcal K(n,R_0,\hat M,\hat \omega)$ has a subsequence which converges in the pointed Gromov-Hausdorff sense to a tuple $(M_\infty,\omega_\infty,p_\infty,\phi_\infty)
 \in \mathcal K(n,R_0,\hat M,\hat \omega)$.
\end{thm}
Here we say a sequence of tuples $(X_i,d_i,p_i,f_i)$ consisting of a pointed metric space $(X_i,d_i,p_i)$ and a map $f_i$ from $X_i$ into a fixed (pointed) metric space converges if both the sequence of spaces $\{(X_i,d_i,p_i)\}$ and the sequence of maps $\{f_i\}$ converges.

A few comments are due.
\begin{itemize}
\item
The assumption about completeness is justified also in the classical Einstein case by the following extension Theorem of Hulin (see also the forthcoming paper \cite{Li23} 
by the second author for various generalizations):
\begin{prop}[\cite{Hul96}]\label{Hulext}
Let $V$ be a connected complex submanifold of $(\mathbb CP^N,\omega_{FS})$ such that the induced K\"ahler metric on $V$ is Einstein. Then $V$ can be extended to a connected injectively immersed complex submanifold $M$ of the same dimension of $V$, such that the induced K\"ahler  (Einstein) metric on $M$ is complete.
\end{prop}

\item
Some conditions are not preserved under the convergence. For example, it is of course well known that even in the case when all $M_i$ are compact, $M_\infty$ does not have to, as shown by all examples of bubbling
known (see \cite{And90} and \cite{S23} for a recent account). See also the explicit example described at the end of this Introduction.

\item
Theorem \ref{cptness2} is in fact a simplified version of the most general result 
 that we obtain (Theorem \ref{cptness1} below in the final Section \ref{finrem} and  which allows the ambient space to be complete and noncompact, and relaxed the assumption
on the lower bound on the scalar curvatures of the $\{(M_i,\omega_i)\}$.

\item
It is interesting to compare this result with the previously regularity results mentioned above. In fact, we are able to prove that under the condition in  Theorem \ref{cptness2}, we can find  positive constants $C, i_0$ and $v$ such that
\begin{enumerate}
\item $|\mathrm{Rm}(\omega_i)|\leq C$  where $\mathrm{Rm}$ denotes the Riemann tensor;
\item The injective radius at any $q\in M_i$ is no less than $i_0$;
\item $\Vol (B(q,1))\geq v$ for any $i$ and $q\in M_i$.
\end{enumerate}
The first estimate can be easily obtained by observing the proof of Lemma \ref{rclb}, while the others follow from Proposition \ref{lccs}.
By works of Gromov (\cite[$\S$ 8. D]{Grmv81}, which was added in the second edition) and Liu (\cite{LG18}), the limit space admits the structure of a $\mathcal C^{1,\alpha}$, for any $\alpha<1$, Riemannian manifold 
(see \cite[Example 5]{Pet87} for an example) and the structure of a complex analytic space respectively. However, this result as well as other regularity results mentioned above are not used in our proof as, for example, we don't know a priori whether the complex structure found by Liu also admits an holomorphic isometry into $(\hat M,\hat\omega)$. 

\item
Donaldson and Sun (\cite{DS14}) showed that, given constants $\kappa,V>0$, the limit space $M_\infty$  of a sequence of compact polarized K\"ahler $n$-manifolds $(M_i,L_i,\omega_i)$ satisfying
\begin{enumerate}
\item $\omega_i\in 2\pi c_1(L_i)$ and $-\omega_i\leq Ric(\omega_i)\leq\omega_i$;
\item $\Vol(B(q,r))>\kappa r^{2n}$ for any $q\in M_i$ and $r>0$;
\item $\Vol(M_i,\omega_i)=V$.
\end{enumerate}
admits a natural  structure of a normal projective variety in $\mathbb CP^N$ for some $N$ depending only on $n,\kappa$ and $M_\infty$. In fact in their proof, they showed that for some sufficiently big $k$, all $M_i$ can be holomorphically  embedded (but not in general isometrically of course)  into some linear subspace of $\mathbb CP^{N_k}$ using the Kodaira embedding maps $\phi_{i,k}$ associated with $H^0(M_i,L_i^k)$, 
and the images converges to a normal subvariety  in $\mathbb CP^{N_k}$ homeomorphic to $M_\infty$.  
As a byproduct of this analysis and our result, one can deduce that every time $M_\infty$ is indeed singular there is no $k\geq 1$ and $R_0\in\mathbb R$ such that  for every $i$, $R(\phi_{i,k}^*\omega_{FS})\geq R_0$, 
indicating the failure of a uniform approximation of Einstein metrics \`a la Tian for a family converging to a singular space.
\end{itemize}

The main applications of our main result concern the study of the following
\begin{prob}\label{conj0}
Let $(M,\omega)$ be a K\"ahler submanifold of $(\mathbb CP^N,\omega_{FS})$ with $N$ finite. If $Ric(\omega)=2\lambda\omega$ for some $\lambda\in\mathbb R$, then $\lambda>0$.
\end{prob}
Problem \ref{conj0} is purely local. However, according to above mentioned extension result by Hulin, we can assume that $(M,\omega)$ is a complete K\"ahler manifold which admits an injective holomorphic isometry $\phi:(M,\omega)\rightarrow\mathbf  (\mathbb CP^N, \omega_{FS})$.

For the case of low codimension, certain local theories provide a positive answer to Problem  \ref{conj0} (see for example \cite{Chern67, Smy67, Tsu86}), while, if $M$ is compact, Hulin \cite{Hul00} found a global argument to prove $\lambda > 0$). However, when the codimension is large or the manifold is non-compact, the same techniques have proved to be powerless.

We will give various applications of Theorem \ref{cptness2} with various curvature assumptions (much milder than the Einstein requirement) towards a solution to the above Problem (for $\lambda=0$). To state our results precisely we set

\begin{defn}
 A non-compact complete connected Riemannian manifold $(M,g)$ has {\em asymptotically non-negative Ricci curvature}, if for any $\varepsilon>0$, $Ric(g)\geq -\varepsilon g$ on $M\!\setminus\! E_\varepsilon$ for some compact set $E_\varepsilon\subset M$. Equivalently for some (hence every) $p\in M$, there exists a non-decreasing function $\kappa:\mathbb R\rightarrow (-\infty,0]$ with $\lim _{t\rightarrow \infty} \kappa(t)=0$, such that
\begin{equation}
Ric(g)\geq (\kappa\circ\rho) g,
\end{equation}
where $\rho$ is the distance function to $p$ on $M$.
\end{defn}

This class of manifolds broadens the class with non-negative Ricci curvature and contains for example all Asymptotically Locally Euclidean spaces of order greater than $2$, which have been
largely studied in connections with various fundamental questions in Geometry and Mathematical Physics (see e.g. \cite{SY}, \cite{LB}, \cite{Kr}, \cite{arpa} and \cite{HLB} for a very limited sample).

\begin{cor}\label{euplane}
Let
$\phi:(M,\omega)\rightarrow (\hat M,\hat\omega)$ be 
a holomorphic isometry from a non-compact complete connected K\"ahler manifold with asymptotically non-negative Ricci curvature.  into a compact K\"ahler manifold $(\hat M,\hat\omega)$.
Then there is a holomorphic isometry from the complex plane into $(\hat M,\hat \omega)$.
\end{cor}
A  well-known example of asymptotically flat  non-compact complex K\"ahler manifold (and hence of  asymptotically non-negative Ricci curvature) is the Burns-Simanca metric $g_{BS}$ on the blow-up 
$\tilde {\mathbb C}^2$ of  ${\mathbb C}^2$ at the origin. In
 \cite[Th. 1]{LOIMOSSAmed} it is proven that there exists a holomorphic isometry
of the complex plane into $(\tilde {\mathbb C}^2, g_{BS})$. 
In view of Corollary \ref{euplane} it could be interesting to see if
there exists a holomorphic isometry of $(\tilde {\mathbb C}^2, g_{BS})$
into some compact  K\"ahler manifold. The only result we know in this direction
is that $(\tilde {\mathbb C}^2, g_{BS})$ cannot admit a holomorphic isometry into
a complex projective space (see \cite{CAL19}).

The strategy of the proof of the above Corollary is inspired by a Claim by Hulin in \cite{Hul96}, according to which if $(M,\omega)$ is complete and Ricci-flat, then its universal cover would actually decompose
as a product of a complex line and a simply connected manifold, as an application of the well known Cheeger-Gromoll  splitting theorem. Unfortunately this Claim turns out to be false in the case that $(M,\omega)$ is non-compact, as the example of the Eguchi-Hanson 
space $T^*S^2$ easily shows. Our idea is then to correct this ill-based hope (and to generalize it to non-necessarily Ricci-flat manifolds) to show that, given a holomorphic isometry 
$\phi:(M,\omega)\rightarrow (\hat M,\hat\omega)$ with $(M,\omega)$ non-compact complete K\"ahler manifold with (asymptotically) non-negative Ricci curvature, we can construct a {\em new} smooth complete manifold with the same curvature bounds and  still 
holomorphic isometrically immersed in the same ambient space as a specific pointed Gromov-Hausdorff limit of well chosen pointed spaces $(M, \omega,p_i)$ converging to $(M_{\infty},\omega_\infty, p_{\infty})$.
By choosing the points $p_i$ carefully, we will show that $(M_{\infty}, \omega_{\infty})$ does indeed  have non-negative Ricci curvature, contains a line and hence its universal covering splits a complex line by Cheeger-Gromoll Theorem. By Theorem \ref{cptness2} this complex line 
is still  holomorphically isometrically immersed in the same ambient space as claimed.

To illustrate explicitly this phenomenon let us consider $M=\mathbb C-\{0\}$ and $\omega=\frac{\iu}{2}\pbp (|z|^2+|z|^{-2})$. One can easily check that  $(M,\omega)$ is complete and has exact two asymptotically Euclidean flat ends.
Clearly $(M,\omega)$ can be holomorphically isometrically embedded in $(\mathbb C^2,\omega_{eucl})$, which in turn maps (by quotient) holomorphically isometrically  onto a complex torus $(\hat M,\hat \omega)$ with a flat metric.
Hence we get a holomorphic isometry $\phi:(M,\omega)\rightarrow (\hat M,\hat \omega)$. Clearly $\{(M,\omega_i,2^i)\}$ converges to $(\mathbb C,\omega_{eucl},0)$ in the pointed Gromov-Hausdorff sense, and this line still maps into the flat torus.

Corollary \ref{euplane} provides a strong obstruction for the existence of holomorphic isometry  for complete non-compact manifolds of non-negative curvature, as it is
well-known that  various spaces do not contain an isometric copy of the complex plane. The most classical of these ambient spaces is in fact the (finite dimensional) projective space $(\mathbb CP^N,\omega_{FS})$
as proved by Calabi in 1953  (\cite[Theorem 13]{Cal53}). Another class of examples are the generalized flag manifold of classical type (see \cite{loimossarighom}).

The following is then an immediate corollary of this discussion:

\begin{cor}\label{nopirf}
There are no holomorphic isometries of non-compact complete K\"ahler manifolds with non-negative Ricci curvature or with asymptotically non-negative Ricci curvature into either  complex finite dimensional projective spaces endowed with Fubini-Study metrics or generalized flag manifolds of classical type.
In particular there is no Ricci-flat non-compact  K\"ahler submanifold of any $(\mathbb CP^N,\omega_{FS})$ with $N$ finite.
\end{cor}
Notice that in the last sentence of the theorem we do not need to  assume
the manifold to be complete by the above mentioned Hulin's extension theorem.

This corollary  proves affirmatively Conjecture 1 in \cite{LSZ18} for finite dimensional ambient spaces and moreover gives a unified proof  (when the ambient projective space is finite dimensional) of various case by case analysis as for example in 
\cite{LZ16, C.A.19, LZZ21, Zed21}.

Clearly a holomorphic isometry 
of non-compact complete K\"ahler manifold into a compact 
 K\"ahler  one cannot be an embedding. Nevertheless 
one can easily construct 
non-compact complete K\"ahler manifolds admitting holomorphic isometric injective immersions into complex finite dimensional projective spaces (see, e.g. Example 6.4 in \cite{loiplacinizedda}). This shows that the assumption on the Ricci curvature in Corollary 
\ref{nopirf} is necessary.

Our results have applications also in other directions. For example, since it is well-known  that an isometry between two  irreducible non Ricci flat K\"ahler manifolds  is either holomorphic or antiholomorphic  (see, e.g \cite{placiniweakly} for a proof)
the last part of Corollary \ref{nopirf} combined with the de Rham decomposition Theorem  yields  the following appealing 
result (to the best of  authors' knowledge this result is new).

\begin{cor}
Let $(M_1, g_1)$ and $(M_2, g_2)$ be two  K\"ahler manifolds. Assume that there exists an isometry between them and $(M_1, g_1)$ admits a holomorphic isometry into a finite dimensional complex projective space. Then $(M_2, g_2)$ also admits a holomorphic isometry into the same finite dimensional complex projective space. 
If additionally $(M_1, g_1)$ is irreducible, then it is biholomorphically isometric to $(M_2, g_2)$ with its original complex structure or its opposite one.
\end{cor}

The paper  is organized as follows. In Section \ref{prel}
we recall basic results on Gromov-Hausdorff convergence.
In   Section \ref{chcm} we prove technical results needed in the proof
of Theorems  \ref{cptness2} and \ref{cptness1}; in particular  we 
show the existence 
of an adapted atlas of charts for a holomorphic isometry between 
K\"ahler manifolds
(see Proposition \ref{lccs}).
Sections \ref{secfour}, \ref{secfive} and \ref{finrem} are dedicated to the proofs of Theorem \ref{cptness2}, Corollary \ref{euplane}
 and Theorem \ref{cptness1},  respectively.

\section{Preliminaries}\label{prel}

\subsection{Conventions}
For any metric space $(X,d)$, $p\in X$ and $r>0$, we define
\begin{equation}
B(p,r)=\{x\in X|d(x,p)<r\},\qquad \bar B(p,r)=\{x\in X|d(x,p)\leq r\}.
\end{equation}
When we want to specify the metric $d$, we use notations like $B_d(p,r)$ nad $\bar B_d(p,r)$. For nonempty $S\subset X$, we also use similar notations $B(S,r)$ and $\bar B(S,r)$.

Especially we denote for any $k\in\mathbb Z_{>0}$ and $\delta>0$
\begin{equation}
B^k_{\delta}=\{x\in\mathbb C^k||x|<\delta\},\qquad \bar B^k_{\delta}=\{x\in\mathbb C^k||x|\leq\delta\}.
\end{equation}

\subsection{Pointed Gromov-Hausdorff convergence}

In the following  all metric spaces that we consider are  boundedly compact, namely all closed bounded subsets  are compact.

First we recall the pointed Gromov-Hausdorff convergence of metric spaces. A pointed metric space is a pair $(X,p)$ consisting of a metric space $X$ and a point $p\in X$. An isometry $f:(X_1,p_1)\rightarrow (X_2,p_2)$ between two pointed metric spaces is an isometry $f:X_1\rightarrow X_2$  with $f(p_1)=p_2$.

The pointed Gromov-Hausdorff convergence can be defined (see e.g. \cite{BBI01}) by means of approximate maps as

\begin{defn}\label{GHspace}
Let $(X_i,p_i)$ ($i=1,2,\cdots$)  and  $(X_\infty,p_\infty)$ be pointed metric spaces. We say that $(X_i,p_i)$  converges to $(X_\infty,p_\infty)$ in the pointed Gromov-Hausdorff sense, if there exists a sequence  $\{\Phi_i\}$ with $\Phi_i:X_\infty\rightarrow X_i$ and  $\Phi_i(p_\infty)=p_i$  satisfying the condition:
\begin{enumerate}[label=(\alph*)]
\item 
for any $r>0$ and $\varepsilon>0$, when $i$ is sufficiently large, $B(\Phi(\bar B(p_\infty,r),\varepsilon)\supset \bar B(p_i,r-\varepsilon)$ and for any $x_1,x_2\in \bar B(p_\infty,r)$
\begin{equation}
|d(\Phi(x_1),\Phi(x_2))-d(x_1,x_2)|<\varepsilon.
\end{equation}
\end{enumerate}
or equivalently there exists a sequence  $\{\Psi_i\}$   with 
$\Psi_i:X_i\rightarrow X_\infty$ and $\Psi(p_i)=p_\infty$ satisfying the condition
\begin{enumerate}[label=(\alph*)]\setcounter{enumi}{1}
\item 
 for any $r>0$ and $\varepsilon>0$,  when $i$ is sufficiently large, $B(\Phi(\bar B(p_i,r),\varepsilon)\supset \bar B(p_\infty,r-\varepsilon)$ and for any $x_1,x_2\in \bar B(p_i,r)$
\begin{equation}
|d(\Psi(x_1),\Psi(x_2))-d(x_1,x_2)|<\varepsilon.
\end{equation}
\end{enumerate}
\end{defn}

In the study of structure of the limit space of a sequence of pointed metric spaces converging in the pointed Gromov-Hausdorff sense, one frequently uses  the idea to construct maps from or into the limit space via finding partial limits of a sequence of maps $\{f_i\}$ of one of the following types
\begin{enumerate}[label= (\Alph*)]
\item $f_i:(X_i,p_i)\rightarrow (Y,q)$;
\item $f_i:(X,p)\rightarrow (Y_i,q_i)$;
\item $f_i:(X_i,p_i)\rightarrow (Y_i,q_i)$.
\end{enumerate}
where $(X,p)$, $(Y,q)$ are  pointed metric spaces, and $\{(X_i,p_i)\}$, $\{(Y_i,q_i)\}$ are sequences of pointed metric spaces which converge in the Gromov-Hausdorff sense to $(X,p)$ and $(Y,q)$ respectively.

In many cases, the sequences involved are approximately Lipschitz in the following sense

\begin{defn}
Let $K\geq 0$ be a constant. We say a sequence $\{f_i\}$ of type (A), (B) or (C) is approximately $K$-Lipschitz  if
for every $r>0$ and $\varepsilon>0$, when $i$ is sufficiently large, it holds that for any $x_1,x_2\in  \bar B(p_i,r)$ (or $x_1,x_2\in \bar B(p,r)$ if $\{f_i\}$ has type (B))
\begin{equation}
d(f_i(x_1),f_i(x_2))\leq Kd(x_1,x_2)+\varepsilon.
\end{equation}
\end{defn}

For approximately Lipschitz sequences, we can consider the following convergence

\begin{defn}\label{GHmap}
Let $\{f_i\}$ be an approximately $K$-Lipschitz  sequence of type (A), (B) or (C) with $K\geq 0$ and $f_\infty:(X,p)\rightarrow (Y,q)$.  We say $\{f_i\}$ converges in the pointed Gromov-Hausdorff sense to $f_{\infty}$ if there exist sequences $\{\Phi_i\}$ and $\{\Psi_i\}$ of maps such that
\begin{enumerate}
\item $\Phi_i=\Id_X$ if $\{f_i\}$ has type (B) or else $\Phi_i:X\rightarrow X_i$ satisfies Condition (a) in Definition \ref{GHspace};
\item $\Psi_i=\Id_Y$ if $\{f_i\}$ has type (A) or else $\Psi_i:Y_i\rightarrow Y$ satisfies Condition (b) in Definition \ref{GHspace};
\item For every $r>0$ and $\varepsilon>0$, when $i$ is sufficiently large, $d(\Psi_i\circ f_i\circ\Phi_i(x),f_{\infty}(x))<\varepsilon$ for any $x\in \bar B(p,r)$.
\end{enumerate}
\end{defn}

\begin{rmk}
The limit map in Definition \ref{GHmap} is unique in the sense that if $f$ and $g$ are two limits of $\{f_i\}$, then $g=\Psi\circ f\circ \Phi$ for some isometries $\Phi:(X,p)\rightarrow (X,p)$ and $\Psi:(Y,p)\rightarrow (Y,p)$. Furthermore, if $\{f_i\}$ has type (A) we can require $\Psi=\Id_Y$ and if $\{f_i\}$ has type (B) we can require $\Phi=\Id_X$.
\end{rmk}

We will use repeatdely the following version of the classical Arzel\`a-Ascoli Lemma, whose proof follows the standard diagonal argument used for the classical one (see for example Petersen \cite{Ped16},  (Lemma 11.1.9)):

\begin{prop}[Arzel\`a-Ascoli]\label{AA} Let $\{f_i\}$ be a sequence of types (A), (B) or (C). If there exists constants $D\geq 0$ and $K\geq 0$ such that
\begin{enumerate}
\item $f_i(p_i)\in \bar B(q,D)$, $f_i(p)\in \bar B(q_i,D)$ or $f_i(p_i)\in\bar B(q_i,D)$according to (A), (B) or (C);
\item $\{f_i\}$ is approximately $K$-Lipschitz.
\end{enumerate}
Then $\{f_i\}$ has a subsequence which converges to a $K$-Lipschitz map $f_\infty:X\rightarrow Y$ with $f_\infty(p)\in \bar B(q,D)$.
\end{prop}

We will also make use of the following foundational Gromov's precompactness Theorem for complete Riemannian manifolds with Ricci curvature bounded from below 

\begin{prop}[Gromov]\label{Grmv}
Let $\{(M_i,g_i,p_i)\}$ be a sequence of pointed complete connected Riemannian manifolds of dimension $n\geq 2$. Assume  there exists a constant $R_0\in\mathbb R$ such that
\begin{equation}
Ric({M_i})\geq (n-1)R_0g_i,
\end{equation}
for every $i$. Then there exists a subsequence of $\{(M_i,g_i,p_i)\}$ which converges in the pointed Gromov-Hausdorff sense to a pointed boundedly compact  length space.
\end{prop}

\section{Choices of holomorphic charts}\label{chcm}

We now introduce a unified method to construct holomoprhic charts on K\"ahler manifolds which admit holomorphic isometric immersions into a fixed K\"ahler manifold. In this Section we are interested in 
geometric properties of a {\em{fixed}} submanifold $M$  that will play a crucial role to  prove the regularity of the limit space of a sequence of such submanifolds in Theorems  \ref{cptness2} and \ref{cptness1}.

Let $(\hat M,\hat \omega)$ be a K\"ahler manifolds of dimension $N\geq 2$. Given a $\hat p\in \hat M$, we can always find an $r_0>0$ and a holomorphic chart map $\psi=(w^1,\cdots,w^N):B(\hat p,2r_0)\rightarrow\mathbb C^N$ such that
\begin{enumerate}
\item $B(\hat p,2r_0)$ is relatively compact in $\hat M$ and $\psi(\hat p)=0$;
\item The coefficients $(\hat g_{\bar b a})$ defined by $\hat\omega=\frac{\iu}{2}\sum\limits_{a,b=1}^N \hat g_{\bar b a}dw^a\wedge d\bar w^b$ satisfies $(\hat g_{\bar b a}(\hat q))=I_N$ and $\frac{1}{2}I_N\leq (\hat g_{\bar b a})\leq 2 I_N$ on $B(\hat p,r_0)$.
\end{enumerate}
For later use, let us also define $C_0\geq0$ and $K_0$  as follows
\begin{equation} \label{defconstants}
C_0^2=\sup_{B(\hat p, r_0)}\sum\limits_{a,b,c=1}^N\left|\frac{\partial \hat g_{\bar b a}}{\partial w^c}\right|^2,\qquad
K_0=\sup_{\hat q\in B(\hat p,r_0)}\left(\sup_{X,Y\in T^{1,0}_{\hat q}\hat M\setminus\{0\}}\frac{\hat R(X,\bar X,\bar Y,Y)}{|X|^2|Y|^2+|\langle X,Y\rangle|^2}\right),
\end{equation}
where $\hat R$ is the Riemann curvature tensor. Note that a straigthforward computation shows that on  $B(\hat p,r_0)$
\begin{equation} \label{Hesspsi}
|\hat\nabla^2\psi|^2=\sum_{a=1}^N|\hat\nabla^2 w^a|^2=4\sum_{a=1}^N\hat g^{b_2\bar b_1}\hat g^{c_2\bar c_1}\overline{\Big(\hat g^{a\bar a_1}\frac{\partial \hat g_{\bar a_1 c_1}}{\partial w^{b_1}}\Big)}\Big(\hat g^{a \bar a_2}\frac{\partial \hat g_{\bar a_2 c_2}}{\partial w^{b_2}}\Big)\leq 64C_0^2.
\end{equation}

Given a positive integer $n<N$ and a real number $R_0\leq 2n(n+1)K_0$, we also set
\begin{equation}  \label{defconstants2}
C=2\sqrt{64C_0^2+2n(n+1)K_0-R_0},\qquad r=\min\{5^{-1}r_0,(16C)^{-1}\}.
\end{equation}
With these notations and this choice of $\psi$ we can now prove the main result of this Section, which shows the existence 
of an adapted atlas of charts for the holomorphically isometrically immersed submanifold, where each chart is defined on a ball of radius controlled by
the scalar curvature lower bound. Moreover we can show that these charts can be chosen in a  way that their Hessians are uniformly controlled and the induced metric is uniformly controlled by the euclidean metrics on the charts, a fact that will play a key role in the proof of the main Theorem \ref{cptness2}.

\begin{prop}\label{lccs}
Let $(M,\omega)$ be a K\"ahler manifold of dimension $n$ and $\phi:(M,\omega)\rightarrow (\hat M,\hat\omega)$ a holomorphic isometry.
 If for some $p\in M$, $\phi(p)\in B(\hat p,r)$, $B(p,4r)$ is relatively-compact in $M$ and $\sup_{B(p,4r)}R(\omega)\geq R_0$, then there exists a holomorphic chart map $\varphi=(z^1,\cdots,z^n)$ on $B(p,2r)$, an $N$-order matrix  $A$, and a holomorphic map $f:\varphi(B(p,2r))\rightarrow \mathbb C^{N-n}$, such that $\phi$ can be locally expressed as
\begin{equation}
\psi\circ\phi\circ\varphi^{-1}(z)=A(z,f(z))+\psi(\phi(p)),
\end{equation}
with, $A^*A=(\hat g_{\bar b a}(\phi(p)))$ and $|f(z)|\leq 2C|z|^2$ on $\varphi(B(p,2r))$.

Moreover the chart $\varphi$ can be chosen to satisfy the additional two properties:
\begin{enumerate}[label=(\alph*)]
\item $\varphi(p)=0$, $|\nabla^2\varphi|\leq C$ and $\frac{1}{2}d(p_1,p_2)\leq |\varphi(p_1)-\varphi(p_2)|\leq 2d(p_1,p_2)$ for any $p_1,p_2\in B(p,2r)$. Consequently $B^n_r\subset\varphi(B(p,2r))\subset B^n_{4r}$.
\item $\omega|_p=\big(\frac{\iu}{2}\sum\limits_{i=1}^ndz^i\wedge d\bar z^i\big)\big|_p$ and $\frac{1}{2}\omega\leq \frac{\iu}{2}\sum\limits_{i=1}^ndz^i\wedge d\bar z^i\leq 2\omega$ on $B(p,2r)$.
\end{enumerate}
\end{prop}

For  reader's convenience, we isolate the following Lemma used in the proof of the above:

\begin{lemma}\label{chart}
Let $p\in M$, $C_1\geq 0$ and $0<r_1\leq (16C_1)^{-1}$. If $B(p,4r_1)$ is relatively-compact in $M$ and $\varphi:B(p,4r_1)\rightarrow \mathbb C^n$ is a holomorphic map such that $(\iu\pbp|\varphi|^2)|_p=2\omega|_p$ and $|\nabla^2\varphi|\leq C_1$. Then $\varphi|_{B(p,2r_1)}$ is a holomorphic chart map and we have the following estimates
\begin{enumerate}
\item $|(\iu\pbp |\varphi|^2)-2\omega|(q)\leq 2C_1d(p,q)$ for any $q\in B(p,4r_1)$.
\item $\frac{1}{2}d(p_1,p_2)\leq |\varphi(p_1)-\varphi(p_2)|\leq 2d(p_1,p_2)$ for any $p_1,p_2\in B(p,2r_1)$.
\end{enumerate}
\end{lemma}
\begin{proof} Since $B(p,4r_1)$ is relatively compact in $M$, any $q\in B(p,4r_1)$ can be connected with $p$ by a minimizing geodesic in $B(p,4r_1)$, and any two distinct $p_1,p_2\in B(p,2r_1)$ can be connected by a minimizing geodesic in $B(p,4r_1)$.

First we give an estimate of $\iu\pbp|\varphi|^2$. Let $\varphi=(\varphi^1,\cdots,\varphi^n)$, then we have
\begin{equation}
|\nabla(\iu\pbp |\varphi|^2)|^2=2\bigg|\sum_{i=1}^n\nabla^2 \varphi^i\otimes d\bar \varphi^i\bigg|^2=2\sum_{i,j=1}^n\langle \nabla^2 \varphi^i,\nabla^2 \varphi^j\rangle \langle d\bar \varphi^i,d\bar \varphi^j\rangle.
\end{equation}
Consider the matrix-valued function $H=(H_{i\bar j})$ defined by
\begin{equation}
H_{i\bar j}=\langle d\bar \varphi^i,d\bar \varphi^j\rangle.
\end{equation}
By using local orthonormal frame fields, one can easily check that
\begin{equation}
|H-2I_n|=|\iu\pbp |\varphi|^2-2\omega|.
\end{equation}
Therefore we have
\begin{equation}\label{pbpxi0}
|\nabla(\iu\pbp |\varphi|^2)|^2\leq (2+|\iu\pbp |\varphi|^2-2\omega|)|\nabla^2\varphi|^2\leq 2C_1^2(2+|\iu\pbp |\varphi|^2-2\omega|).
\end{equation}
Let  $\theta=\iu\pbp |\varphi|^2-2\omega$. Clearly we have
\begin{equation}
\theta|_p=0,\qquad |\nabla\theta|\leq C_1\sqrt{4+2|\theta|}.
\end{equation}
For any $q\in B(p,4r_1)$, let $t_0=d(p,q)$ and $\alpha:[0,t_0]\rightarrow B(p,4r_1)$ a geodesic with $|\alpha'|=1$, $\alpha(0)=p$ and $\alpha(t_0)=q$. Consider $f=|\theta|\circ\alpha$, then $f$ is Lipschitz and
\begin{equation}
f(0)=0,\qquad f'\leq C_1\sqrt{4+2f}.
\end{equation}
By solving the differential inequality, we have
\begin{equation}
|\theta|(q)=f(t_0)\leq 2^{-1}(2+C_1t_0)^2-2\leq \frac{17}{32}.
\end{equation}
where the last inequality holds as $C_1t_0\leq 4C_1r_1\leq 4^{-1}$.
Consequently we have on $B(p,4r_1)$ 
\begin{equation}\label{pbpxi}
\frac{47}{32}\omega\leq \iu\pbp |\varphi|^2\leq \frac{81}{32}\omega.
\end{equation}

Next we give an upper bound of $|\varphi(p_1)-\varphi(p_2)|$ for $p_1,p_2\in B(p,4r_1)$. For any real tangent vector $X\in T_{q}M$ with $q\in B(q,4r_1)$, let $X^{1,0}$ be the $(1,0)$ part of $X$, then the holomorphicity of $\varphi$ yields
\begin{equation}\label{dphi0}
|X\varphi|^2=|X^{1,0}\varphi|^2=(\pbp|\varphi|^2)(X^{1,0},\overline{X^{1,0}})\leq \frac{81}{32}|X^{1,0}|^2=\frac{81}{64}|X|^2.
\end{equation}
Since $p_1$ and $p_2$ can be connected with $p$ by geodesics of length less that $4r_1$ in $B(p,4r_1)$ respectively, \eqref{dphi0} implies
\begin{equation}\label{dphi1}
|\varphi(p_1)-\varphi(p_2)|\leq |\varphi(p_1)-\varphi(p)|+|\varphi(p_2)-\varphi(p)|\leq 2\times 4r_1\times \frac{9}{8}=9r_1.
\end{equation}

Finally, we give an estimate of $|\varphi(p_1)-\varphi(p_2)|$ for $p_1,p_2\in B(p,2r_1)$. To obtain a upper bound, noting that $p_1$ and $p_2$ can be connected by a geodesic of length $d(p_1,p_2)$, \eqref{dphi0} implies
\begin{equation}
|\varphi(p_1)-\varphi(p_2)|\leq \frac{9}{8}d(p_1,p_2)\leq 2d(p_1,p_2),
\end{equation}
hence getting the second inequality in Part $(2)$ of the Lemma.
To obtain a lower bound, we consider the following function defined on $B(p,4r_1)$
\begin{equation}
\xi=|\varphi-\varphi(p_1)|^2.
\end{equation}
Clearly we have
\begin{equation}\label{xi0}
\xi(p_1)=0,\qquad d\xi|_{p_1}=0.
\end{equation}
At the same time, noting that $\varphi$ is holomorphic, we have
\begin{equation}
\iu\pbp \xi=\iu\pbp |\varphi|^2,\qquad|\nabla^{2,0} \xi|\leq |\nabla^2\varphi||\varphi-\varphi(p_1)|.
\end{equation}
Then by \eqref{pbpxi} and \eqref{dphi1}, we have
\begin{equation}
\iu\pbp \xi\geq\frac{47}{32}\omega,\qquad |\nabla^{2,0}\xi|\leq \frac{9}{16}.
\end{equation}
Consequently, for any real tangent vector $X\in T_{q}M$ with $q\in B(q,4r)$, let $X^{1,0}$ be the $(1,0)$ part of $X$, then
\begin{equation}\label{hessxi1}
\nabla^2\xi(X,X)=2\nabla^2\xi(X^{1,0},\overline {X^{1,0}})-2{\rm Re}\nabla^2\xi(X^{1,0},X^{1,0})\\
\geq\frac{29}{16} |X^{1,0}|^2=\frac{29}{32}|X|^2.
\end{equation}
Let $t_0=d(p_1,p_2)$ and $\gamma:[0,t_1]\rightarrow B(p,4r)$ a geodesic such that $|\gamma'|=1$, $\gamma(0)=p_1$ and $\gamma(t_1)=p_2$. Consider $c:[0,t_1]\rightarrow \mathbb R$ defined as
\begin{equation}
c=\xi\circ\gamma.
\end{equation}
By \eqref{xi0} and \eqref{hessxi1}, we have
\begin{equation}
c(0)=0,\qquad c'(0)=0,\qquad c''\geq \frac{29}{32}.
\end{equation}
By Taylor expansion we get
\begin{equation}\label{dphi3}
|\varphi(p_2)-\varphi(p_1)|^2=c(t_0)\geq \frac{29}{64}t_0^2=\frac{29}{64}d(p_1,p_2)^2 \geq \frac{1}{4}d(p_1,p_2)^2,
\end{equation}
hence getting the first inequality in Part $(2)$ of the Lemma.

Since $\varphi|_{B(p,2r_1)}$ is injective and $d\varphi_q$ is nonsingular for any $q\in B(p,4r_1)$, $\varphi|_{B(p,2r_1)}$ is a holomorphic chart map.
\end{proof}

\begin{proof}[Proof of Proposition \ref{lccs}]
Let $V=d\phi_p(T^{1,0}_pM)$ and $V^{\perp}$ the orthogonal complement of $V$ in $T^{1,0}_{\phi(p)}{\hat M}$. Since
\begin{equation}\{\sqrt{2}{\textstyle\frac{\partial}{\partial w^1}}|_{\phi(p)},\cdots,\sqrt{2}{\textstyle\frac{\partial}{\partial w^N}}|_{\phi(p)}\},\end{equation}
is a basis of $T^{1,0}_{\phi(p)}\hat M$, we can choose an $N$-order matrix $T=(T_a^b)$ such that
\begin{equation}\{\sqrt{2}T_1^a{\textstyle\frac{\partial}{\partial w^a}}|_{\phi(p)},\cdots,\sqrt{2}T_n^a{\textstyle\frac{\partial}{\partial w^a}}|_{\phi(p)}\},\end{equation}
is an orthonormal basis of $V$ and
\begin{equation}\{\sqrt{2}T_{n+1}^a{\textstyle\frac{\partial}{\partial w^a}}|_{\phi(p)},\cdots,\sqrt{2}T_N^a{\textstyle\frac{\partial}{\partial w^a}}|_{\phi(p)}\}\end{equation}
is an orthonormal basis of $V^{\perp}$.

Let $A=T^{-1}$. By the choice of $T$, we have $T^*(\hat g_{\bar b a}(\phi(p)))T=I_N$. Then 
$A^*A=(g_{\bar b a}(\phi(p)))$ and $TT^*=(\hat g^{a \bar b}(\phi(p)))$. Consequently
\begin{equation}
|A^*A-I_N|\leq 2C_0d(\phi(p),\hat p)\leq 256^{-1}.
\end{equation}
Let $\varphi=(z^1,\cdots,z^n):B(p,4r)\rightarrow\mathbb C^n$ and $h:B(p,r_0)\rightarrow\mathbb C^{N-n}$ be the holomorphic maps defined by
\begin{equation}
(\varphi,h)=T(\psi\circ\phi-\psi(\phi(p))).
\end{equation}
Then $\psi\circ\phi=A(\varphi,h)+\psi(\phi(p))$ and one can easily check that
\begin{equation}
\varphi(p)=0,\qquad \omega|_p=\bigg(\frac{\iu}{2}\sum\limits_{i=1}^ndz^i\wedge d\bar z^i\bigg)\bigg|_p,\qquad h(p)=0,\qquad \nabla h(p)=0. 
\end{equation}

We now observe that, thanks to Lemma \ref{chart}, $\varphi|_{B(p,2r)}$ is a holomorphic chart map which satisfies properties $(a)$ and $(b)$. On one hand  we have
\begin{equation}
|\nabla^2\varphi|^2+|\nabla^2 h|^2=|\nabla^2 (T\psi\circ\phi)|^2\leq (1-256^{-1})^{-1}|\nabla^2(\psi\circ\phi)|^2.
\end{equation}

On the other hand one can prove the following inequality (see Lemma \ref{lemmabelow} below for a proof).

 \begin{equation} \label{hesspsiphi}
|\nabla^2(\psi\circ\phi)|^2\leq 128C_0+4n(n+1)K_0-2R_0.
\end{equation}

By combining these inequalities we get

\begin{equation}
|\nabla^2\varphi|^2+|\nabla^2 h|^2\leq 256C_0+8n(n+1)K_0-4R_0=C^2.
\end{equation}
and then by Lemma \ref{chart}, $\varphi|_{B(p,2r)}$ is  a holomorphic chart map and the estimates in statements  $(a)$ and $(b)$. Besides, let $f=h\circ(\varphi|_{B(p,2r)})^{-1}$, then clearly $\psi\circ\phi\circ\varphi=A(z,f(z))+\psi(\phi(p))$. To complete the proof, we only need to show that $|f(z)|\leq 2C|z|^2$ on $\varphi(B(p,2r))$, which is an easy direct consequence of $h(p)=0$, $\nabla h(p)=0$ and $|\nabla^2 h|\leq C$.
\end{proof}

\begin{lemma}\label{lemmabelow}
Estimate \eqref{hesspsiphi} holds true.
\end{lemma}
\begin{proof}
As $\phi$ is $1$-Lipschitz, we have for $q\in B(p,4r)$
\begin{equation}
d(\phi(q),\hat p)\leq d(\phi(q),\phi(p))+d(\phi(p),\hat q)\leq d(q,p)+d(\phi(p),q)<5r\leq r_0.
\end{equation}
This implies that $\phi(B(p,4r))\subset B(\hat p,r_0)$. Let now $\{\hat e_1,\cdots,\hat e_N\}$ be the holomorphic frame field of $\phi^*(T^{1,0}\hat M)$ over $B(p,r_0)$ defined by
\begin{equation}
\hat e_a(q)=\left(q,{\textstyle \frac{\partial}{\partial w^a}}|_{\phi(q)}\right).
\end{equation}
We can treat the second fundamental form $B$ of $\phi:(M,\omega)\rightarrow (\hat M,\hat\omega)$ as a smooth section of $\phi^*(T^{1,0}\hat M)\otimes (\Lambda^{1,0}M\otimes \Lambda^{1,0}M)$. On $B(p,r_0)$, we can write
\begin{equation}
B=B^a\hat e_a,
\end{equation}
where each $B^a$ is a smooth section of $\Lambda^{1,0}M\otimes \Lambda^{1,0}M$. For $a=1,\cdots,N$, by the definition of $B$ and $\nabla^2(w^a\circ\phi)$, one can easily check that
\begin{equation}
\nabla^2(w^a\circ\phi)=\phi^*(\hat\nabla^2w^a)-B^a,
\end{equation}
and consequently
\begin{equation}
|\nabla^2(w^a\circ\phi)|^2\leq 2|\hat\nabla^2w^a|^2\circ\phi+2|B^a|^2.
\end{equation}
By inequality \eqref{Hesspsi}, we  have
\begin{equation}
\sum_{a=1}^N|\phi^*(\nabla^2w^a)|^2\leq |\hat \nabla^2\psi|^2\circ\phi\leq 64C_0^2.
\end{equation}
To complete the proof of inequality  \eqref{hesspsiphi} we only need to give estimates on $\sum\limits_{a=1}^N|B^a|^2$:
by the Gauss-Codazzi equation, we have
\begin{equation}
R(X,\bar X,\bar Y,Y)=(\phi^*\hat R)(X,\bar X,\bar Y,Y)-\langle B(X,Y),B(X,Y)\rangle,
\end{equation}
for any $q\in M$ and $X,Y\in T^{1,0}M$, and $R$ and $\hat R$ are the Riemannian tensors of $\omega$ and $\hat\omega$ respectively.
Let $\{e_1,\cdots,e_n\}$ be a smooth local orthonormal frame field of $T^{1,0}M$, then we have
\begin{equation}
|B|^2=\sum_{i,j=1}^n[(\phi^*\hat R)(e_i,\bar e_i,\bar e_j,e_j)-R(e_i,\bar e_i,\bar e_j,e_j)]\leq n(n+1)K_0-\frac{1}{2}R_0,
\end{equation}
Consequently, we have
\begin{equation}
\sum_{a=1}^N|B^a|^2\leq 2 (\hat g^{a \bar b}\circ\phi) \langle B^a,B^b\rangle=2|B|^2\leq 2n(n+1)K_0-R_0.
\end{equation}
This concludes the proof of the estimate  \eqref{hesspsiphi}. 
\end{proof}

\section{Proof of Theorem \ref{cptness2}}\label{secfour}
The first step of the proof of Theorem \ref{cptness2} is to find a pointed Gromov-Hausdorff limit as a length space. For this we first find upper and lower bounds on curvatures of elements of $ \mathcal K(n,R_0,\hat M,\hat \omega)$ (see Definition \ref{deffond}).

Set $K:\hat M\rightarrow \mathbb R$ to be the function defined by
\begin{equation}\label{dfek0}
K(\hat q)=\sup_{X,Y\in T^{1,0}_{\hat q}\hat M\setminus\{0\}}\frac{\hat R(X,\bar X,\bar Y,Y)}{|X|^2|Y|^2+|\langle X,Y\rangle|^2},
\end{equation}
and notice that $K_0=\sup_{\hat q\in B(\hat p,r_0)}\{K(\hat q) \mid \hat q \in \hat M\}$ from equation \eqref{defconstants}.

\begin{lem}\label{rclb}
Let $\phi:(M,\omega)\rightarrow (\hat M,\hat \omega)$ be a holomorphic isometry  with $(M,\omega)$ an $n$-dimension K\"ahler manifold and $(\hat M,\hat\omega)$ an $N$-dimensional K\"ahler manifold. Then for any $q\in M$ and nonzero $X,Y\in T^{1,0}_qM$, we have
\begin{equation}\label{rclbeq0}
\frac{1}{4}R(\omega)(q)-\frac{n^2+n-2}{2}K(\phi(q))\leq \frac{R(X,\bar X,\bar Y,Y)}{|X|^2|Y|^2+|\langle X,Y\rangle|^2}\leq  K(\phi(q)).
\end{equation}

\end{lem}
\begin{proof}
We denote by $R$ and $\hat R$ the Riemannian tensor of $\omega$ and $\hat\omega$ respectively. We also denote by $B$ the second fundamental form of $\phi:(M,\omega)\rightarrow (\hat M,\hat\omega)$, which can be  treated as a smooth section of $\phi^*(T^{1,0}\hat M)\otimes (\Lambda^{1,0}M\otimes \Lambda^{1,0}M)$. 
Given a $q\in M$, by the Gauss-Codazzi equation, we have for any $X,Y\in T_q^{1,0}M$
\begin{equation}
R(X,\bar X,\bar Y,Y)=(\phi^*\hat R)(X,\bar X,\bar Y,Y)-\langle B(X,Y),B(X,Y)\rangle.
\end{equation}
Using the the definition of $K$, we have for any $X,Y\in T_q^{1,0}M$
\begin{equation}\label{rclbeq1}
R(X,\bar X,\bar Y,Y)\leq K(\phi(q))(|X|^2|Y|^2+|\langle X,Y\rangle|^2).
\end{equation}
We define $S^2(T^{1,0}_qM)$ to be the symmetric square of $T^{1,0}_qM$.  The map
\begin{equation}
(X\otimes Y,Z\otimes W)\mapsto R(X,\bar Z,\bar W,Y),
\end{equation}
induces a Hermitian $2$-form on $S^2(T^{1,0}_qM)$, which we denote by $H$. At the same time, the map
\begin{equation}
(X\otimes Y,Z\otimes W)\mapsto \langle X, Z\rangle \langle Y, W\rangle+ \langle X, W\rangle  \langle Y, Z\rangle.
\end{equation}
induces an inner product on $S^2(T^{1,0}_qM)$, which we denote by $G$. Let $H^{\#}$ be the linear endomorpism of $S^2(T^{1,0}_qM)$ defined by
\begin{equation}
G(H^{\#}u,v)=H(u,v),\qquad \forall u,v\in S^2(T^{1,0}_qM).
\end{equation}
Then $H^{\#}$ is Hermitian with respect to $G$, and
\begin{equation}
\tr H^{\#}=\frac{1}{4}R(\omega)(q).
\end{equation}
Furthermore, we have by \eqref{rclbeq1}
\begin{equation}
H^{\#}\leq K(\phi(q))\Id.
\end{equation}
Therefore we have
\begin{equation}
H^{\#}\geq\left(\frac{1}{4}R(\omega)(q)-\frac{n^2+n-2}{2}K(\phi(q))\right)\Id,
\end{equation}
which indicates that
\begin{equation}
R(X,\bar X,\bar Y,Y)\geq \left(\frac{1}{4}R(\omega)(q)-\frac{n^2+n-2}{2}K(\phi(q))\right)(|X|^2|Y|^2+|\langle X,Y\rangle|^2),
\end{equation}
for any $X,Y\in T^{1,0}_qM$.
\end{proof}

In particular it follows from  Lemma \ref{rclb} that, if $\hat M$ is compact and $R(\omega)\geq R_0$ for some constant, then the Ricci curvature $Ric(\omega)\geq -2n(n+1)K_1\omega$ for some constant $K_1$ depending only on $n, R_0$ and $(\hat M,\hat\omega)$. By this fact, Proposition \ref{Grmv} and Proposition \ref{AA}, we have
\begin{cor}\label{lmtls}
Under the condition of Theorem \ref{cptness2}, there exists a subsequence of $\{(M_i,\omega_i,p_i,\phi_i)\}$ which converges in the pointed Gromov Hausdorff sense to some $(X_\infty,d_\infty,p_\infty,\phi_\infty)$ consisting of a pointed boundedly compact length space $(X_\infty,d_\infty,p_\infty)$ and a $1$-Lipschitz map $\phi_\infty$.
\end{cor}

Our task is then to study  the regularity and the various riemannian properties of the pointed Gromov-Hausdorff limit $(X,d, p,\phi)$  we just obtained as a limit of a sequence $\{(M_i,\omega_i,p_i,\phi_i)\}$ of type $\mathcal K(n,R_0,\hat M,\hat\omega)$. 
This is achieved by the following

\begin{prop}\label{regularity}
$X$ admits a natural structure of a complex manifold such that $\phi$ is a holomorphic immersion and $d$ coincides with the distance induced by $\phi^*\hat\omega$.
\end{prop}

To prove Proposition \ref{regularity}, we need the following 
\begin{lem}\label{lcappr}
For any $q\in X$, there exists an $r>0$,  an open embedding map $\varphi:B(q,r)\rightarrow\mathbb \mathbb \mathbb C^n$ with $\varphi(0)=0$ and $\varphi(B(p,r))\subset B^n_r$, a holomorphic chart map $\psi:B(\phi(q),8r)\rightarrow \mathbb C^N$ and a holomorphic embedding $\eta:B^n_{2r}\rightarrow \mathbb C^N$ with $\eta(0)=0$, such that
\begin{enumerate}
\item $\psi\circ\phi\circ\varphi^{-1}=\eta$ on $\varphi(B(q,r))$.
\item Let $\tilde \omega=(\psi^{-1}\circ\eta)^*\hat\omega$ and $d_{\tilde\omega}$ the induced distance on $(B^n_{4r},\tilde\omega)$, then $\varphi(B(q,r))=B_{\tilde\omega}(0,r)$ and for any $q',q''\in B(q,r)$, it holds that
\begin{equation}
d_{\tilde\omega}(\varphi(q'),\varphi(q''))=d(q',q'').
\end{equation}
\end{enumerate}
Furthermore, there exists subsequence $\{l_k\}$ of $\{1,2,\cdots\}$ and for each $k\geq 1$ a point $q_{l_k}\in M_{l_k}$ with $|d(q_{l_k},p_{l_k})-d(q,p)|\leq 4^{-k}r$, an open neighborhood $U_{l_k}\subset B(q_{l_k},4r)$ of $q_{l_k}$  and a holomorphic chart map $\varphi_{l_k}:U_{l_k}\rightarrow \mathbb C^n$ with $\varphi_{l_k}(0)=0$ and $\varphi_{l_k}(U_{l_k})=B^n_{2r}$, such that
\begin{enumerate}\setcounter{enumi}{2}
\item  On $B^n_{2r}$, $\{\psi\circ\phi_{l_k}\circ\varphi_{l_k}^{-1}\}$ and $\{\varphi_{l_k}^*\omega_{l_k}\}$ locally converge to $\eta$ and $\tilde\omega$ in $\mathcal C^\infty$ topology respectively.
\end{enumerate}
\end{lem}
Each of the statement of the above Lemma will serve a very clear scope in the process of proving Proposition \ref{regularity} and in turn Theorem \ref{cptness2}: indeed, part (1) will imply that $X$ is indeed a smooth complex manifold and $\phi$ is indeed a holomorphic immersion, while (2) will be the key to show 
that $(X,d)$ is indeed a K\"ahler manifold and $\phi$  indeed preserves the metric (namely $d$ coincides with the distance induced by $\phi^*\hat\omega$, as explained below). As for (3), it will allow us to obtain certain riemannian bounds of the limit space. For example, in the proof of Theorem \ref{cptness2}, it will allow  the limit space to have the same (scalar) curvature bounds as the the $(M_i, \omega_i)$, and in the proof of Corollary \ref{euplane}, it will also allow the limit space to have have non-negative Ricci curvature.

\begin{proof}[Proof Proposition \ref{regularity}]
Clearly $X$ is Hausdorff and seperable.
First endow $X$ with a suitable holomorphically compatible atlas.
Let $\mathcal A$ be the set of all pairs $(U,\varphi)$ of a nonempty set $U$ together with a map $\varphi:U\rightarrow \mathbb C^n$ which satisfy the condition that $\phi$ is chart map and $\phi\circ\varphi^{-1}$ is a holomorphic embedding. 
By Lemma \ref{lcappr} (1),for every $q\in X$, we can find a $(U,\varphi)\in\mathcal A$ with $U=B(q,r)$ for some $r>0$. Therefore we have $\bigcup_{(U,\varphi)\in\mathcal A}U=X$. At the same time, if $(U_1,\varphi_1),(U_2,\varphi_2)\in\mathcal A$ and $V=U_1\cap U_2\neq\emptyset$, then $\phi(V)$ is a complex submanifold of $\hat M$ of dimension $n$. Clearly $\varphi_1\circ(\phi|_V)^{-1}$ and $\varphi_2\circ(\phi|_V)^{-1}$ are two holomorphic chart map on $V$, so $\varphi_1\circ\varphi_2^{-1}=\left(\varphi_1\circ(\phi|_V)^{-1}\right)\circ\left(\varphi_2\circ(\phi|_V)^{-1}\right)^{-1}$ is holomorphic on $\varphi_2(V)$. Therefore $\mathcal A$ is a complex atlas of $X$.

Next we show that $\phi$ is a holomorphic immersion. This is evident since for any $(U,\varphi)\in \mathcal A$, $\phi\circ\varphi^{-1}$ is a holomorphic immersion.

Finally we show that the metric $d$ on $X$ coincides with the metric $d_{\phi^*\hat\omega}$ induced by $\phi^*\hat \omega$ on $X$. According to Lemma \ref{lcappr} (2), for any $q\in M$, there exists an $r>0$ such that for  any $q',q''\in B(q,r)$
\begin{equation}
d(q',q'')=d_{\phi^*\hat\omega}(q',q'').
\end{equation}
Noting that $d$ and $d_{\phi^*\hat\omega}$ are intrinsic, we have $d=d_{\phi^*\hat\omega}$.
\end{proof}

\begin{proof}[Proof of Lemma \ref{lcappr}]

Given $\hat q=\phi(q)$, we choose $\hat q=\phi(q)$,  $r_0>0$, a holomorphic chart map $\psi=(w^1,\cdots,w^N)$ on $(\hat q,4r_0)$, $C_0\geq 0$ and $K_1$ as in Definitions  \eqref{defconstants},  \eqref{Hesspsi} and  \eqref{defconstants2} in Section 3.
We also  set $r_1=d(q,p)+r_0$.

Next we choose an approximating sequence of $(\bar B(q,4r),d, q,\phi|_{\bar B(q,4r)})$.
Since $(X,d, p,\phi)$ is the limit of $\{(M_l,\omega_l,p_l,\phi_l)\}$, we can choose a subsequence $\{l_k\}$ of $\{1,2,\cdots\}$, and  for each $k$, a map $\Phi_{l_k}:\bar B(p_{l_k}, r_1)\rightarrow M_{l_k}$ such that
\begin{itemize}
\item $B(\Phi_{l_k}(\bar B(p, r_1)),4^{-k-1}r)\supset \bar B(\Phi_{l_k}(p),r_1)$ and it holds that for any $q',q''\in \bar B(q,r_1)$
\begin{equation}
|d(\Phi_{l_k}(q'),\Phi_{l_k}(q''))-d(q',q'')|<4^{-k-1}r.
\end{equation}
\item For any $q'\in \bar B(p,r_1)$, it holds that
\begin{equation}
d(\phi_{l_k}\circ\Phi_{l_k}(q'),\phi(q'))<4^{-k-1}r.
\end{equation}
\end{itemize}
Let $q_{l_k}=\Phi_{l_k}(q)$ for each $k$. One can easily check that
\begin{enumerate}[label=(\roman*)]
\item $|d(q_{l_k},p_{l_k})-d(q,p)|<4^{-k-1}r$ and $d(\phi_{l_k}(q),\hat q)<4^{-k-1}r$. 
\end{enumerate}
And for any $s\in(0,8r]$\begin{equation}
\Phi_{l_k}(\bar B(q,s)),4^{-k}r)\supset \bar B(q_{l_k},s)
\end{equation}

Third we construct a holomorphic embedding $\eta:B^n_{2r}\rightarrow\mathbb C^N$.
Since $(M_{l_k},\omega_{l_k})$ is complete and $\phi_{l_k}$ is a holomorphic isometry, $B(q_{l_k},8r)$ is relatively-compact in $M_{l_k}$ and $\phi_{l_k}(B(q_{l_k},8r))\subset B(\hat q,r_0)$. At the same time, $\inf_{B(q_{l_k},8r)}R(\omega_{l_k})\geq R_0$. Applying Proposition \ref{lccs}, we can find an open neighborhood $U_{l_k}$ of $q_{l_k}$ and a chart map $\varphi_{l_k}$ on $U_{l_k}$, such that
\begin{enumerate}[label=(\roman*)]\setcounter{enumi}{1}
\item $\varphi_{l_k}(q_{l_k})=0$, $\varphi_{l_k}(U_{l_k})=B^n_{2r}$ and $\psi\circ\phi_{l_k}\circ\varphi_{l_k}^{-1}(z)=A_{l_k}(z,f_{l_k}(z))+\psi(\phi(q_{l_k}))$ for some $N$-order matrix $A_{l_k}$ with ${A_{l_k}}^*A_{l_k}=(g_{\bar b a}(\phi_{l_k}(q_{l_k})))$ and holomorphic map $f_{l_k}:B^n_r\rightarrow \mathbb C^{N-n}$ with $|f_{l_k}(z)|\leq (16r)^{-1}|z|^2$.
\item\label{estw} The metric $\tilde\omega_{l_k}=(\varphi_{l_k}^{-1})^*\omega_{l_k}$ satisfies
$\frac{\iu}{4}\sum\limits_{i=1}^ndz^i\wedge d\bar z^i\leq\tilde\omega_{l_k}\leq \iu\sum\limits_{i=1}^ndz^i\wedge d\bar z^i$.
\end{enumerate}
Since ${A_{l_k}}^*A_{l_k}\rightarrow I_N$ and $|f_{l_k}|\leq (16r)^{-1}|z|^2$, we can find a subsequence $\{l'_k\}$ of $\{l_k\}$, such that
\begin{enumerate}[label=(\roman*)]\setcounter{enumi}{3}
\item $A_{l_k}$ converges to an $N$-order unitary matrix $A$ and $f_{l'_k}$ locally converges in $\mathcal C^{\infty}$ topology to a holomorphic map $f:B^n_{4r}\rightarrow \mathbb C^n$ with $|f(z)|\leq (16r)^{-1}|z|^2$.
\end{enumerate}
Let $\eta:B^n_{2r}\rightarrow B^N_{4r}$ be the map defined by $\eta(z)=A(z,f(z))$, $\tilde\omega=(\psi^{-1}\circ\eta)^*\hat\omega$. By the choices of $A$ and $f$, together with the fact that $\psi(\phi_{l'_k}(q_{l'_k}))$ converges to $\psi(\hat q)=0$, we have
\begin{itemize}
\item $\psi\circ\phi_{l'_k}\circ\varphi_{l'_k}^{-1}$ locally  converges to $\eta$ in $\mathcal C^\infty$ topology.
\item $\tilde\omega_{l'_k}$ locally converges to $\tilde\omega$ in $\mathcal C^\infty$ topology 
(this is due to
the fact that $\varphi_{l'_k}$ is a holomoprhic isometry).
\end{itemize}
Together with \ref{estw}, we have
\begin{enumerate}[label=(\roman*)]\setcounter{enumi}{4}
\item\label{estw1} $\frac{\iu}{4}\sum\limits_{i=1}^ndz^i\wedge d\bar z^i\leq\tilde\omega\leq \iu\sum\limits_{i=1}^ndz^i\wedge d\bar z^i$.
\end{enumerate}

Fourth we construct an embedding map $\varphi:B(p,r)\rightarrow \mathbb C^n$.
Let $d_{\tilde\omega}$ be the induced metric on $(B^n_{2r},\tilde\omega)$ and $d_{\tilde\omega_{l_k}}$  the induced metric on $(B^n_{2r},\tilde\omega_{l_k})$. By \ref{estw} and \ref{estw1}, when $s\in (0,\sqrt{2})$, the metric balls $B_{\tilde\omega}(0,s)$ in $(B^n_{2r},\tilde\omega)$ and $B_{\tilde\omega_{l_k}}(0,s)$  in $(B^n_{2r},\tilde\omega_{l_k})$ are relatively-compact in $B^n_{2r}$. Together with the fact that $\varphi_{l,k}:(U_{l_k},\omega_{l_k})\rightarrow (B^n_{2r},\tilde\omega_{l_k})$ is a holomorophic isometry,  we have $\varphi_{l_k}(\bar B(q_{l_k},s))=\bar B_{\tilde\omega_{l_k}}(0,s)$ and  for any $q',q''\in \bar B(q_{l_k},s)$
\begin{equation}
d_{\tilde\omega_{l_k}}(\varphi_{l_k}(q'),\varphi_{l_k}(q''))=d(q',q'').
\end{equation}
Consider the map $\sigma_{l_k}=\varphi_{l_k}\circ\Phi_{l_k}$ defined on $\bar B(q,r)$. One can easily check that
\begin{itemize}
\item $\sigma_{l_k}(q)=0$ and $B_{\tilde\omega_{l_k}}(\sigma_{l_k}(\bar B(q,r)),4^{-k}r)\supset \bar B_{\tilde\omega_{l_k}}(0,r)$.
\item For any $q',q''\in \bar B(q,r)$, it holds that
\begin{equation}
|d_{\tilde\omega_{l_k}}(\Phi_{l_k}(q'),\Phi_{l_k}(q''))-d(q',q'')|<4^{-k-1}r.
\end{equation}
\end{itemize}
Since $\tilde\omega_{l'_k}$ locally converges to $\tilde\omega$ in $\mathcal C^\infty$ topology, we can find a positive sequence $\epsilon_k\rightarrow 0$, such that 
\begin{itemize}
\item $B_{\tilde\omega}(\sigma_{l'_k}(\bar B(q,r)),\epsilon_k)\supset \bar B_{\tilde\omega}(0,r)$.
\item For any $q',q''\in \bar B(q,2r)$, it holds that
\begin{equation}
|d_{\tilde\omega}(\sigma_{l'_k}(q'),\sigma_{l'_k}(q''))-d(q',q'')|<\epsilon_k.
\end{equation}
\end{itemize}
Together with the compactness of $\bar B(q,r)$ and applying the Arzel\`a-Ascoli lemma, we can find a subsequence $\{l''_k\}$ of $\{l'_k\}$, such that $\sigma_{l''_k}$ uniformly converges to a map  $\varphi:\bar B(q,r)\rightarrow B^n_{2r}$. Clearly
\begin{enumerate}[label=(\roman*)]\setcounter{enumi}{5}
\item $\varphi(q)=0$, $\varphi(\bar B(q,r))=\bar B_{\tilde\omega}(q,r)$ and for any $q',q''\in B(q,r)$,  it holds that
\begin{equation}
d_{\tilde\omega}(\varphi(q'),\varphi(q''))=d(q',q'').
\end{equation}
\end{enumerate}
So $\varphi|_{B(p,r)}$ is an open embedding, thus a chart map.

Finally we check that
\begin{enumerate}[label=(\roman*)]\setcounter{enumi}{6}
\item $\psi\circ\phi\circ\varphi^{-1}=\eta$ on $\bar B_{\tilde\omega}(0,r)$. 
\end{enumerate}
In fact we have on $\bar B_{\tilde\omega}(0,r)$
\begin{equation}
\psi\circ(\phi_{l''_k}\circ\Phi_{l''_k})\circ\varphi^{-1}=(\psi\circ\phi_{l''_k}\circ\varphi_{l''_k}^{-1})\circ(\varphi_{l''_k}\circ\Phi_{l''_k})\circ\varphi^{-1}.
\end{equation}
Noting that $\phi_{l''_k}\circ\Phi_{l''_k}$ uniformly converges to $\phi$ on $\bar B(q,4r)$, $\psi\circ\phi_{l''_k}\circ\varphi_{l''_k}^{-1}$ locally uniformly converges to $\eta$ on $B^n_{2r}$ and $\varphi_{l''_k}\circ\Phi_{l''_k}=\sigma_{l''_k}$ uniformly converges to $\varphi$ on $\bar B(q,r)$. So we have on $\bar B_{\tilde\omega}(0,r)$
\begin{equation}
\psi\circ\phi\circ\varphi^{-1}=\eta\circ\varphi\circ\varphi^{-1}=\eta.
\end{equation}
This concludes the proof of Lemma \ref{lcappr}.
\end{proof}

Collecting all the above we can complete the proof of our main Theorem:

\begin{proof}[Proof of Theorem \ref{cptness2}]
First, by Corollary \ref{lmtls} we know that $\{(M_i,\omega_i,p_i,\phi_i)\}$ converges in the Gromov-Hausdorff sense to a tuple $(X_\infty,d_{\infty},p_\infty,\phi_{\infty})$ which consists of a pointed  boundedly-compact length space $(X_\infty,d_\infty,p_\infty,\phi_\infty)$ and a $1$-Lipschitz map $\phi:(X_\infty,d_\infty)\rightarrow (\hat M,\hat\omega)$. Secondly, by Proposition \ref{regularity} we can endow $(X_\infty,d_\infty)$ with the structure of a K\"ahler manifold, denoted by $(M_\infty,\omega_{\infty})$, such that the distance induced by $\omega_\infty$ coincides with $d_\infty$, and $\phi:(M_\infty,\omega_\infty)\rightarrow (\hat M,\hat \omega)$ is a holomorphic isometric immersion. Finally, by  Lemma \ref{lcappr} (3) $\omega_{\infty}$ is locally the limit of a sequence $\{\omega_k'\}$ of K\"ahler metrics with $R(\omega_k')\geq R_0$. Therefore $R(\omega_{\infty})\geq R_0$ and consequently $(M_\infty,\omega_\infty,p_\infty,\phi_{\infty})\in\mathcal K(n,R_0,\hat M,\hat\omega)$.

\end{proof}

\section{Proof of Corollary \ref{euplane}} \label{secfive}

Before giving the proof of Corollary \ref{euplane}, we recall the Cheeger-Gromoll splitting theorem (\cite[Theorem 2]{CG71}): a line in a complete Riemannian manifold $(M,g)$ is a geodesic $\gamma:\mathbb R\rightarrow M$ satisfying $|\gamma'|=1$ and $d(t_1,t_2)=|t_1-t_2|$.  If $(M,g)$ is a complete Riemannian manifold of dimension $m\geq 2$ with non-negative Ricci curvature which contains a line, then $(M,g)$ is the product of a Riemannian manifold of dimension $m-1$ and a real line $\mathbb R$. If in addition  $(M,\omega)$ is K\"ahler and $m=2n$, then the universal covering of $(M,\omega)$ is the product of a K\"ahler manifold of complex dimension $n-1$ and a  complex line $\mathbb C$.
\begin{proof}[Proof of Corollary \ref{euplane}]
We need to show if there is a holomorphic isometry $\phi:( M, \omega)\rightarrow (\hat M,\hat\omega)$ with $( M, \omega)$ a complete K\"ahler manifold with non-negative Ricci curvature and then there is holomorphic isometry form a complex line into $(\hat M,\hat\omega)$.

As mentioned in the Introduction, we want to show that by choosing an appropriate sequence $p_i\in M$, $(M,\omega,p_i, \phi)$ converges to some $(M_\infty,\omega_\infty,  p_\infty, \phi_\infty)$ that contains a line. 

Since $(M,\omega)$ is complete and  non-compact, we can find a geodesic $\gamma:\mathbb R\rightarrow M$ such $|\gamma'|=1$ and $\gamma|_{[0,\infty)}$ is a ray in the sense that
\begin{equation}
d(\gamma(t_1),\gamma(t_2))=|t_1-t_2|,\qquad \forall t_1,t_2\in[0,\infty).
\end{equation}
Let $p_i=\gamma(2^i)$, then we get a sequence of pointed K\"ahler manifold $\{(M,\omega,p_i)\}$. We can apply Theorem \ref{cptness2} in both cases considered as the condition $Ric(\omega)\geq (\kappa\circ\rho) \omega$ clearly gives a lower bound on the scalar curvature too, so, up to passing to a subsequence, we can assume that $\{M,\omega, p_i,\phi\}$ converges in the pointed Gromov-Hausdorff sense to a $(M_\infty,\omega_\infty,p_\infty,\phi_\infty )$ which consists of a complete K\"ahler manifold $(M_\infty,\omega_\infty)$ 
 of dimension $n$ (possibly different from $(M,\omega)$) and a holomorphic isometry $\phi_\infty:(M_\infty,\omega_\infty)\rightarrow (\hat M,\hat \omega)$.
 
Note that if $(M, \omega)$ has non-negative Ricci curvature, the same holds for $(M_\infty,\omega_\infty)$ by following the argument which we gave in the proof of Theorem \ref{cptness2} to obtain  the lower bound for the scalar curvature. On the other hand, if $(M, \omega)$ has asymptotically non-negative Ricci curvature, we claim that $Ric(\omega_\infty)\geq 0$ again.
Indeed, let $q_\infty\in M_\infty$, then by  Lemma \ref{lcappr} (3), we can find an $r>0$, a subsequence $\{l_k\}$ of $\{1,2,\cdots\}$, a sequence $\{q_{l_k}\}$ of points in $M$ and a sequence $\mu_k$ of biholomorphic maps  $\mu_k:B(q_\infty,r)\rightarrow U_{l_k}$ with $\mu_k(q_\infty)=q_{l_k}$ and $U_{l_k}\subset B(q_{l_k},4r)$ an open neighborhood of $q_{l_k}$, such that
\begin{enumerate}[label=(\alph*)]
\item $|d(q_{l_k},p_{l_k})-d(q_\infty,p_\infty)|<4^{-k}$.
\item On $B(q_\infty,r)$, $\mu_{k}^*\omega$ converges to $\omega_\infty$ in $\mathcal C^\infty$ topology.
\end{enumerate}
 If $Ric(\omega)\geq 0$, then obviously (b) implies $Ric(\omega_\infty)\geq 0$, while if $Ric(\omega)$ has  asymptotically non-negative curvature, then by (a) we have for $q\in B(q_{l_k},4r)$
\begin{equation}
d(p,q))\geq d(p,p_{l_k})-d(p_{l_k},q_{l_k})-d(q_{l_k},q)\geq 2^{l_k}-1-4r
\end{equation}
Together with $Ric(\omega)\geq(\kappa\circ\rho)\omega$, we have $Ric(\omega)\geq \kappa(2^{l_k}-1-4r)\omega$ on $B(q_{l_k},4r)$ and consequently
\begin{equation}
Ric(\mu_k^*\omega)\geq \kappa(2^{l_k}-1-4r)\mu_k^*\omega.
\end{equation}
on $B(q_\infty,r)$. Using (b) and $\lim _{t\rightarrow\infty} \kappa(t)=0$ we deduce that $Ric(\omega_\infty)\geq 0$.

Consider the sequence  of geodesics $\gamma_i:\mathbb R\rightarrow (M,\omega,p_i)$ defined by $\gamma_i(t)=\gamma(t+2^i)$. Clearly each $\gamma_i$ is $1$-Lipschitz and $\gamma_i(0)=p_i$. By Proposition \ref{AA}, up to passing to a subsequence, we can assume that $\{\gamma_i\}$ converges in the Gromov-Hausdorfff sense to a $1$-Lipschitz map $\gamma_\infty:\mathbb R\rightarrow (M_\infty,\omega_\infty)$ with $\gamma_\infty(0)=p_\infty$. Furthermore, for any $t_1,t_2\in \mathbb R$, when $2^{i}\geq \max\{-t_1,-t_2\}$, we have
\begin{equation}
d(\gamma_i(t_1),\gamma_i(t_2))=|t_1-t_2|.
\end{equation}
Letting $i\rightarrow\infty$, we obtain $d(\gamma_\infty(t_1), \gamma_\infty(t_2))=|t_1-t_2|$. Together with the fact that $\gamma_\infty$ is $1$-Lipschitz, it is a  line.
Since 
$Ric(\omega_\infty)\geq 0$ the Corollary follows  by Cheeger-Gromoll.
\end{proof}

\section{Final remarks}\label{finrem}
Theorem \ref{cptness2} can be easily generalized in two directions, firstly to study immersions into complete K\"ahler ambient spaces $(\hat M,\hat\omega)$, not necessarily compact, and secondly to relax the 
scalar curvature lower bound asking that it satisfies $R(\omega)\geq R_0\circ\rho$, where $R_0$ is a non-increasing function and $\rho$ is the distance function on $M$ from $p$. We believe that this latter could lead to some interesting consequence in looking for uniform versions of Tian's approximation Theorem as mentioned in the Introduction.

The most general result we prove is the following:

 \begin{thm}\label{cptness1}
Let $S\subset \hat M$ be a nonempty compact set. If $\{(M_i,\omega_i,p_i,\phi_i)\} \in \mathcal K(n,R_0,\hat M,\hat\omega)$ and 
$\phi_i(p_i)\in S$, then a subsequence converges to some $(M_\infty,\omega_\infty,p_\infty,\phi_\infty)\in \mathcal K(n,R_0,\hat M,\hat\omega)$, with $\phi_\infty (p_\infty) \in S$.
\end{thm}

Noting that $\sup_{\bar B(p,r)} R(\omega)$ is increasing, it is natural to require $R_0$ to be non-increasing, so that the condition $R(\omega)\geq R_0\circ\rho$ is equivalent to $\inf_{\bar B(p,r)} R(\omega)\geq R_0(r)$ for any $r\geq 0$.

The following example show that  $R(\omega_i)$ could  be  not uniformly bounded from below by a  constant $R_0$ but each $R(\omega_i)$ is bounded by $2n(n+1)K\circ\rho_i$ for some non-increasing function $K:[0,\infty)\rightarrow \mathbb R$.

\begin{eg} Let $f$ be an entire holomorphic function on $\mathbb C$ with $f(0)=0$. Assume the Taylor series of $f$ at $0$ is $f(z)=\sum\limits_{i=1}^\infty a_iz^i$. For each $i=1,2,\cdots$, consider
\begin{equation}
f_i(z)=\sum_{j=1}^ia_jz^j,\qquad \omega_i={\frac{\iu}{2}}\pbp(|z|^2+|f_i|^2),\qquad \phi_i(z)=(z,f_i(z)).
\end{equation}
Then when $i\rightarrow \infty$, $\omega_i$ locally uniformly $\mathcal C^{\infty}$ converges to $\omega_{\infty}\triangleq\omega_i={\textstyle\frac{\iu}{2}}\pbp(|z|^2+|f|^2)$ and $\phi_i$  locally uniformly $\mathcal C^{\infty}$ converges to $\phi_{\infty}(z)\triangleq(z,f(z))$. For some $f$, for example $f(z)=e^{z^2}-1$, $\inf_{\mathbb C} R(\omega_\infty)=-\infty$. In this case, clearly $\inf_{i\geq 1}\inf_{\mathbb C}R(\omega_i)=-\infty$ . However, we always have
\begin{equation}
\frac{1}{4}R(\omega_i)\geq -|f''|^2\geq K(|z|)\triangleq -\sum_{j=2}^\infty j(j-1)|a_j||z|^{2(j-2)},
\end{equation}
which implies $R(\omega_i)\geq K\circ\rho_i$, where $\rho_i$ is the distance function to $0$ on $(\mathbb C,\omega_i)$.
\end{eg}

The proof of Theorem \ref{cptness1} follows the same idea of Theorem 
\ref{cptness2}.

First we need the following generalization of Gromov's precompactness Theorem for complete Riemannian manifolds with Ricci curvature bounded from below, which is likely well known to the experts but we add a proof for reader's convenience:
\begin{prop}\label{gGrmv}
Let $\{(M_i,g_i,p_i)\}$ be a sequence of  pointed complete connected Riemannian manifold of dimension $n\geq 2$. If there is a non- increasing function $R_0:[0,\infty)\rightarrow\mathbb R$, such that
\begin{equation}
Ric(g_i)\geq (n-1)(R_0\circ\rho_i)g_i,
\end{equation}
for each $i$, where $\rho_{i}$ is the distance function to $p_i$. Then there exists a subsequence of $\{(M_i,g_i,p_i)\}$ which converges in the Gromov-Hausdorff sense to a pointed boundedly compact length space.
\end{prop}
\begin{proof}
By \cite[Theorem 8.1.10]{BBI01} (the idea is similar to \cite[Proposition 5.2]{Grmv81}), we only need to show that for each $r>0$ and $\varepsilon>0$ , there exists a positive integer $N=N(r,\varepsilon)$, such for each $i$, there is a finite set  $S_i\subset\bar B(p_i,r)$ consisting of at most $N$ points and satisfying $B(S_i,\varepsilon)\supset \bar B(p_i,r)$. The condition on the Ricci curvature indicates that $Ric(g_i)\geq (n-1)R_0(2r)$ on $\bar B(p_i 2r)$. According to \cite[End of the proof of 5.3]{Grmv81}, we can choose
\begin{equation}
N=\left\lceil\frac{V(n,R_0(2r),2r)}{V(n,R_0(2r),\varepsilon)}\right\rceil,
\end{equation}
where $\lceil\cdot \rceil$ is the ceil function and $V(n,\kappa,t)$ is the volume of a geodesic ball of radius $t$ in a simply connected space form of dimension $n$ and sectioncal curvature $\kappa$.
\end{proof}

Secondly,  by Lemma \ref{rclb}, for any $(M,\omega,p,\phi)\in\mathcal K(n,R_0,\hat M,\hat \omega)$ we can find a non-increasing function $R_0:[0,\infty)\rightarrow \mathbb R$ such that $Ric(\omega)\geq 2n(n+1)(R_0\circ\rho)\omega$. Thus under the condition of Theorem \ref{cptness1} and the compactness of $S$, we can apply Proposition \ref{gGrmv} and Lemma \ref{AA} to  find a subsequence of $\{(M_i,\omega_i,p_i,\phi_i)\}$ converging to some $(X_\infty,d_\infty,p_\infty,\phi_\infty)$ which consists of a pointed boundedly compact length space $(X_\infty,d_\infty,p_\infty)$ and a $1$-Lipschitz map $\phi$ with $\phi_\infty(p_\infty)\in S$.
All the remaining proofs follow the same arguments with the following adaptations:
 (1) when choosing the constant $C$ in \eqref{defconstants2}, we use $R_0(r_0)$ instead of $R_0$; (2) when verifying the condition of  Proposition \ref{lccs}, we have $\inf_{B(q_{l_k},8r)}R(\omega_{l_k})\geq R_0(r_0)$ instead of  $\inf_{B(q_{l_k},8r)}R(\omega_{l_k})\geq R_0$.
 By doing this we get the same conclusion as in Theorem \ref{cptness2}, replacing the constant $R_0$ with a function of $\rho$, that we still denote by $R_0$.

\end{document}